\documentclass[reqno, 11pt]{amsart}

\usepackage[
	a4paper,
	margin=1in
]{geometry}
\usepackage{setspace}

\pdfoutput=1
\makeatletter
\def\section{\@startsection{section}{1}%
  \z@{4.5ex \@plus 1ex \@minus .2ex}%
  {1.5ex \@plus .2ex}%
  {\normalfont\large\scshape\centering\S}}
\def\subsection{\@startsection{subsection}{2}%
  \z@{2.5ex \@plus .5ex \@minus .2ex}%
  {0.5ex \@plus .2ex}%
  {\normalfont\large\bfseries}}
\makeatother

\usepackage{appendix}

\usepackage{amsmath, amssymb, amsthm}
\usepackage{mathtools}

\numberwithin{equation}{section}

\usepackage[
	setpagesize=false,
	colorlinks=true,
	linkcolor=BrickRed,
        citecolor=OliveGreen,
        urlcolor=black,
	pdfencoding=auto,
	psdextra,
]{hyperref}
\usepackage{cleveref}

\crefname{claim}{claim}{claims}
\Crefname{claim}{Claim}{Claims}

\usepackage{etoolbox} 
\usepackage[shortlabels]{enumitem}
\usepackage{xparse} 
\usepackage{blkarray,multirow}

\usepackage{graphicx}
\usepackage[dvipsnames]{xcolor}
\usepackage{tikz}
\usetikzlibrary{patterns}
\usetikzlibrary{decorations.markings, arrows.meta}

\usepackage{todonotes}
\usepackage[
    deletedmarkup=sout,
    commentmarkup=uwave,
    authormarkuptext=name
]{changes}

\usepackage{comment}

\setlist[enumerate,1]{label=(\arabic*), ref=(\arabic*)}
\setlist[enumerate,3]{label=(\roman*), ref=(\roman*)}

\theoremstyle{plain}
\newtheorem{theorem}{Theorem}[section]
\newtheorem{lemma}[theorem]{Lemma}
\newtheorem{corollary}[theorem]{Corollary}
\newtheorem{proposition}[theorem]{Proposition}
\newtheorem{observation}[theorem]{Observation}
\newtheorem{conjecture}[theorem]{Conjecture}
\newtheorem{problem}[theorem]{Problem}

\newtheorem{remark}{Remark}

\newtheorem{claim}[theorem]{Claim}
\newtheorem*{claim*}{Claim}

\makeatletter
\newenvironment{claimproof}[1][Proof]{\par
	\pushQED{\qed}%
	
	\normalfont \topsep6\p@\@plus6\p@\relax
	\trivlist
	\item[\hskip\labelsep
	\textit{#1}\@addpunct{.}~]\ignorespaces
}{%
	\popQED\endtrivlist\@endpefalse
}
\makeatother

\theoremstyle{definition}
\newtheorem{definition}[theorem]{Definition}

\newcommand{\calD}{\mathcal{D}}
\newcommand{\calF}{\mathcal{F}}

\newcommand{\calL}{\mathcal{L}}

\newcommand{\defeq}{\coloneqq}

\newcommand{\Hm}{\mathrm{Hom}}
\newcommand{\hm}{\mathrm{hom}}

\title{On Sidorenko exponents of hypergraphs}

\author[Hyunwoo Lee]{Hyunwoo Lee}

\usepackage[square,sort,comma,numbers]{natbib}
\begin{document}
\maketitle

\begin{abstract}
    For an $r$-graph $F$, define Sidorenko exponent $s(F)$ as $$s(F):= \sup \{s \geq 0: \exists \text{$r$-graph $H$ s.t. } t_F(H) = t_{K^{(r)}_r} (H)^s > 0\},$$ where $t_{H_1}(H_2)$ denotes the homomorphism density of $H_1$ in $H_2$. The celebrated Sidorenko's conjecture states that $s(F) = e(F)$ holds for every bipartite graph $F$.
    It is known that for all $r \geq 3$, the $r$-uniform version of Sidorenko's conjecture is false, and only a few hypergraphs are known to be Sidorenko.
    
    In this paper, we discover a new broad class of Sidorenko hypergraphs and obtain general upper bounds on $s(F)$ for certain hypergraphs related to dominating hypergraphs. This makes progress toward a problem raised by Nie and Spiro. We also discover a new connection between Sidorenko exponents and upper bounds on the extremal numbers of a large class of hypergraphs, which generalizes the hypergraph analogue of K\H{o}v\'{a}ri--S\'{o}s--Tur\'{a}n theorem proved by Erd\H{o}s.
\end{abstract}


\section{Introduction}\label{sec:intro}

\subsection{Background}\label{subsec:background}

A \emph{homomorphism} from an $r$-graph $F$ to an $r$-graph $H$ is a function $\phi: V(F) \to V(H)$ such that for all $e\in E(F)$, its image $\phi(e)$ is an edge of $H$. Denote $\Hm(F, H)$ and $\hm(F, H)$ by the set of homomorphisms from $F$ to $H$ and its size, respectively.
A \emph{homomorphism density} of $F$ in $H$, denoted $t_F(H)$, is defined as follows.
$$
    t_F(H) \defeq \frac{\hm(F, H)}{v(H)^{v(F)}}.
$$
A central topic in extremal (hyper)graph theory is to determine the minimum value of $t_F(\cdot)$ for some (hyper)graph $F$ under various additional assumptions. For instance, a celebrated \emph{Sidorenko's conjecture}~\cite{Sidorenko1,Sidorenko2} states that for every bipartite graph $F$, the value $t_F(G)$ is always larger than or equal to the homomorphism density of $F$ in a random graph with the same edge density as $G$. Formally, the following is the statement of Sidorenko's conjecture.

\allowdisplaybreaks
\begin{conjecture}[Sidorenko~\cite{Sidorenko1,Sidorenko2}]\label{conj:sidorenko}
    For every bipartite graph $F$ and every graph $G$, the following inequality holds.
    \begin{equation*}
        t_F(G) \geq t_{K^{(2)}_2}(G)^{e(F)}.
    \end{equation*}
\end{conjecture}

In the last two decades, there have been extensive studies to prove that certain classes of bipartite graphs satisfy Sidorenko's conjecture~\cite{Conlon-Fox-Sudakov,CKLL,Finite-reflection,Sidorenko-blowup,Coregliano,Coregliano-Razborov,KLL,Li-Szegedi,Szegedy,ILL,CLM,Hatami}, but it still remains widely open.

One natural generalization of \Cref{conj:sidorenko} is to consider its hypergraph version. One may ask whether for every $r$-partite $r$-graph $F$, the following inequality holds for all $r$-graphs $H$:
\begin{equation}
    t_F(H) \geq t_{K_r^{(r)}}(H)^{e(F)}.\label{eq:sidorenko}
\end{equation}

Unfortunately, for every $r \geq 3$, this is not true~\cite{Sidorenko2,Conlon-Lee-Sidorenko}. For instance, a loose $3$-uniform triangle does not satisfy \eqref{eq:sidorenko}. Hence, it is natural to quantify the extent to which a given $r$-partite $r$-graph $F$ deviates from being Sidorenko. To do this, we define \emph{Sidorenko exponent} for $r$-partite $r$-graph $F$ as follows.

\begin{definition}\label{def:sidorenko exponent}
    Let $F$ be an $r$-partite $r$-graph. Then we define the Sidorenko exponent of $F$, denoted $s(F)$, as follows.
    \begin{equation*}
        s(F):= \sup\{s \geq 0: \exists \text{$r$-graph $H$ s.t. $t_F(H) = t_{K_r^{(r)}}(H)^s > 0$} \}.
    \end{equation*}
    If $s(F) = e(F)$, then we say $F$ is \emph{Sidorenko}.
\end{definition}

We note that by considering a random $r$-graph, the inequality $s(F) \geq e(F)$ always holds for all $r$-partite $r$-graph $F$.

Recently, determining the value of $s(F)$ for various $r$-partite $r$-graph $F$ has started to receive attention from many researchers. For instance, Mandache~\cite{Mandache} and Fox, Sah, Sawhney, Stoner, and Zhao~\cite{Triforces} proved that the Sidorenko exponent for $3$-uniform loose triangle is $4$. Conlon, Lee, and Sidorenko~\cite{Conlon-Lee-Sidorenko} recently proved that if $s(F) > e(F)$, then one can significantly improve the lower bound on the extremal number of $F$ that was obtained from the probabilistic deletion method. Here, the \emph{extremal number} of $r$-graph $F$, denoted by $\mathrm{ex}(n, F)$, is the maximum number of edges of $r$-graphs that do not contain a sub-hypergraph isomorphic to $F$. 
Consequently, Nie and Spiro~\cite{Nie-Spiro} adapt the method of Conlon, Lee, and Sidorenko to improve the lower bound for the random Tur\'{a}n problem whenever $s(F) > e(F)$. They also determined and found upper and lower bounds on $s(F)$ for expanded hypergraphs, and raised a question to determine the Sidorenko exponents of various hypergraphs\footnote[1]{Nie and Spiro also used the same notation $s(F)$, but their definition of $s(F)$ is smaller than ours by $e(F)$.} (see also~\cite{super-odd-lin-cycle}).

In this paper, we find a new large class of Sidorenko hypergraphs and also obtain reasonable upper bounds on Sidorenko exponents of certain hypergraphs. We also discovered a new connection to Sidorenko exponents and upper bounds on the extremal numbers of certain hypergraphs.


\subsection{Main theorems}\label{subsec:results}

Our first main theorem is a hypergraph generalization of the following remarkable result.

\begin{theorem}[Conlon--Fox--Sudakov~\cite{Conlon-Fox-Sudakov}]\label{thm:Conlon-Fox-Sudakov}
    Every bipartite graph that has a vertex complete to the other part is Sidorenko.
\end{theorem}

For the sake of convenience, from now on, we consider all hypergraphs to be labelled hypergraphs. Before we state our main theorem, we first define several notations. For an $r$-graph $F$ and a vertex subset $U\subseteq V(F)$, we denote by $d_F(U)$ the number of edges of $F$ that contain at least one vertex of $U$.
Let $F$ be an $r$-partite $r$-graph on vertex partition $V_1 \cup \cdots V_r$. Let $v\in V_i$ be a vertex where $i\in [r]$. We define a \emph{link hypergraph} of $v$ in $F$, denoted $\calL_F(v)$ as 
$$
    V(\calL_F(v)) \defeq \{u\in \bigcup_{j\in [r]} V_j \setminus V_i: \exists e\in E(F) \text{ s.t. } \{u, v\}\subseteq e \}
$$ 
and 
$$
    E(\calL_{F}(v)) \defeq \{f\in V(\calL_F(v))^{(r-1)}): f\cup \{v\} \in E(F)\}.
$$ 
Note that link hypergraphs of $r$-graphs are $(r-1)$-graphs.
Suppose $(v_1, \dots, v_t)$ be the vertices of the $r$-th part of $F$. Then we define \emph{link profile} of $F$ as a tuple of link hypergraphs $(\calL_F(v_1), \dots, \calL_F(v_t))$.

Our first main theorem is the following.
\begin{theorem}\label{thm:main-sidorenko}
    Let $r \geq 2$ be an integer and $M$ be a disjoint union of Sidorenko $(r-1)$-graphs, say $S_1, \dots, S_{\ell}$.
    Let $F$ be an $r$-graph with link profile $(\calL_F(v_1), \dots, \calL_F(v_t))$ such that $\calL_F(v_1) = M$. Assume for each $i \in [t]$, the link hypergraph $\calL_F(v_i)$ is a sub-hypergraph of $M$ such that $\calL_F(v_i)$ is a disjoint union of some members in $\{S_1, \dots, S_{\ell}\}$. Then $F$ is Sidorenko. 
\end{theorem}

\begin{remark}
    Note that every finite set of vertices is a Sidorenko $1$-graph, so the $r = 2$ case of \Cref{thm:main-sidorenko} is exactly the same as \Cref{thm:Conlon-Fox-Sudakov}.
\end{remark}

Observe that \Cref{thm:main-sidorenko} can produce a Sidorenko hypergraph from other Sidorenko hypergraphs with smaller uniformity. Hence, \Cref{thm:main-sidorenko} provides a new large class of Sidorenko hypergraphs inductively, and also gives a good upper bound on the Sidorenko exponents for certain hypergraphs.

One caveat of \Cref{thm:main-sidorenko} is that the link hypergraphs should be a disjoint union of Sidorenko hypergraphs. Hence, it is hard to apply for bounding Sidorenko exponents of hypergraphs with connected links. To overcome this problem, our second main theorem provides a more general upper bound on Sidorenko exponents of hypergraphs.

In order to state our theorem, we need a notion of \emph{dominating hypergraphs}.

\begin{definition}\label{def:dominating hypergraph}
    Let $r \geq 1$ and $F$, $F'$ be $r$-graphs. We say $F$ \emph{dominates} $F'$ if the following inequality holds for all $ r$-graphs $H$.
    \begin{equation*}
        t_F(H)^{1/e(F)} \geq t_{F'}(H)^{1/e(F')}.
    \end{equation*}
    If $F$ dominates all of its sub-hypergraphs with at least one edge, then we say $F$ is a \emph{dominating hypergraph}.
\end{definition}

Dominating hypergraphs are closely related to Gowers norms, norming, and weakly norming hypergraphs. Because of its strong domination properties in terms of homomorphism counting, it was revealed that the notion of dominating hypergraph plays a central role in extremal combinatorics. We refer the reader to~\cite{Finite-reflection,dominating-graphs,Hatami,Hatami-thesis} for more information about dominating hypergraphs.

We are now ready to state our second main theorem.

\begin{theorem}\label{thm:main-domination}
    Let $r \geq 2$ be an integer and $M$ be a dominating $(r-1)$-graph.
    Let $F$ be an $r$-graph with link profile $(\calL_F(v_1), \dots, \calL_F(v_t))$ such that $\calL_F(v_1) = M$ and for all $i \in [t]$, the link hypergraph $\calL_F(v_i)$ is a sub-hypergraph of $M$. Then we have the following inequality.
    \begin{equation*}
        s(F) \leq \sum_{i\in [t]} d_M(V(\calL_F(v_i))).
    \end{equation*}
\end{theorem}

Our last main theorem is on a connection to Sidorenko exponents and an upper bound on the extremal number, as we promised in \Cref{subsec:background}. This provides another motivation to find upper bounds on Sidorenko exponents of hypergraphs. 

Let $F$ be an $(r-1)$-graph and $t > 0$ be a positive integer. We denote by $F(t)$ the $r$-graph such that its $r$-th part has $t$ vertices and its link profile is $(F, \dots, F)$, and inductively define $F(t_1, \dots, t_{\ell}) \defeq F(t_1 ,\dots, t_{\ell-1})(t_{\ell})$. Note that a complete $r$-partite $r$-graph $K^{(r)}_{t_1,\dots,t_{r}}$ is the same as $K^{(r-1)}_{t_1, \dots, t_{r-1}}(t_r) = K^{(2)}_{t_1, t_2}(t_3, \dots, t_r)$. Among the fundamental results in extremal combinatorics, there are K\H{o}v\'{a}ri--S\'{o}s--Tur\'{a}n theorem and Erd\H{o}s' theorem, which gives an upper bound on the extremal number of $K^{(r)}_{t_1, \dots, t_{r}}$.

\begin{theorem}[K\H{o}v\'{a}ri--S\'{o}s--Tur\'{a}n~\cite{Kovari-Sos-Turan} $(r = 2)$, Erd\H{o}s~\cite{Erdos-KST} $(r \geq 3)$]\label{thm:KST}
    Let $r \geq 2$ be a positive integer. Then for all $t > 0$, the following holds.
    \begin{equation*}
        \mathrm{ex}(n, K^{(r)}_{t_1, \dots, t_{r}}) \leq O_{t_1, \dots, t_{r}} \left(n^{r - \frac{1}{\prod_{i\in [r-1]}t_i}}\right).
    \end{equation*}
\end{theorem}

Our last main theorem is a generalization of \Cref{thm:KST}.

\begin{theorem}\label{thm:main-extremal}
    Let $r > \ell \geq 1$ be integers and $F$ be an $(r-\ell)$-graph. Then for all $t_1, \dots, t_{\ell} > 0$, the following holds.
    \begin{equation*}
        \mathrm{ex}(n, F(t_1, \dots, t_{\ell})) \leq O_{F, t_1, \dots, t_{\ell}}\left(n^{r - \frac{1}{s(F)\prod_{i\in [\ell - 1]} t_i}}\right).
    \end{equation*}
\end{theorem}

Note that the exponent of $n$ in the inequality in \Cref{thm:main-extremal} does not depend on $t_{\ell}$. One direct corollary of \Cref{thm:main-extremal} with a probabilistic deletion method is the following.

\begin{corollary}\label{cor:extremal-sidorenko}
    Let $F$ be an Sidorenko $(r-1)$-graph. Then we have the following.
    \begin{equation*}
       n^{r - \frac{1}{e(F)} - o_t(1)} \leq \mathrm{ex}(n, F(t)) \leq O_{F, t}(n^{r - \frac{1}{e(F)}}).
    \end{equation*}
\end{corollary}

\begin{remark}
    Since complete bipartite graphs are Sidorenko, \Cref{cor:extremal-sidorenko} implies \Cref{thm:KST}. Indeed, our proof gives an alternative proof of \Cref{thm:KST}. The original proof of \Cref{thm:KST} uses a double counting argument with induction, but our proof explicitly uses Sidorenko-type inequalities for homomorphism counting, which allows a new connection between the Sidorenko exponents and extremal numbers of partite hypergraphs. 
\end{remark}

By combining all our main theorems, for certain hypergraphs, one can considerably improve the upper bounds of extremal numbers of them upon the bounds given by \Cref{thm:KST}. We collect several applications of our main theorems in \Cref{sec:applications}.
\newline

\paragraph{\textbf{Organization.}}
The rest of the paper is organized as follows. In \Cref{sec:applications}, we discuss several applications of our main theorems, and most of them will be proved in \Cref{sec:proof-applications}. In \Cref{sec:prelim}, we clarify our notations and definitions that we use throughout this paper. We also demonstrate the tensor power tricks and (weakly) norming, dominating hypergraphs in \Cref{sec:prelim}. The proofs of our main theorems are provided in \Cref{sec:proof}. Finally, we discuss further directions and some open problems in \Cref{sec:concluding}.

\section{Applications}\label{sec:applications}
In this section, we collect several applications of our main theorems discussed in \Cref{subsec:results}. Most of our applications are focused on bounding Sidorenko exponents of certain families of hypergraphs, including $3$-uniform tight cycles. We also obtained several general bounds for sparse hypergraphs. 

Denote $C^{(r)}_k$ by the $r$-uniform tight cycle of length $k$, that is $V(C^{(r)}_k) = [k]$ and edges are in form of $\{i, i+1, \dots, i+r-1\}$ for all $i\in [k]$, where addition is taken mod $k$. Recently, Conlon, Lee, and Sidorenko~\cite{Conlon-Lee-Sidorenko} shown that for all $\ell \geq 2$, the $3$-partite $3$-uniform tight cycle $C^{(3)}_{3\ell}$ are not Sidorenko $3$-graph.
\begin{theorem}[Conlon--Lee--Sidorenko~\cite{Conlon-Lee-Sidorenko}]\label{thm:tightcycle-not-sidorenko}
    For all integers $\ell \geq 2$, the following inequality holds.
    \begin{equation*}
        s\left(C^{(3)}_{3\ell} \right) > 3\ell.
    \end{equation*}
    In other words, every $3$-partite $3$-uniform tight cycles of length at least $6$ is a non-Sidorenko hypergraph.
\end{theorem}

Motivated by \Cref{thm:tightcycle-not-sidorenko}, the natural question is how the $3$-uniform tight cycles are far from \eqref{eq:sidorenko}.
By using \Cref{thm:main-domination}, we bound the value of $s\left(C^{(3)}_{3\ell}\right)$ up to a multiplicative constant.

\begin{theorem}\label{thm:tight-3-cycle}
    For all $\ell \geq 2$, we have
    \begin{equation*}
        s\left(C^{(3)}_{3\ell}\right) \leq 7\ell.
    \end{equation*}
\end{theorem}

For general hypergraphs, a naive approach to obtaining an almost trivial bound on Sidorenko exponents is to use the fact that complete $r$-partite $r$-graphs are Sidorenko. Let $F$ be an $r$-graph on vertex partition $V_1 \cup \cdots \cup V_r$ with $|V_i| = t_i$. Since $F$ is a sub-hypergraph of $K^{(r)}_{t_1, \dots, t_r}$ and $K^{(r)}_{t_1, \dots, t_r}$ is Sidorenko, one can deduce that $s(F) \leq \prod_{i\in [r]} t_i$. Our next theorem shows that if $F$ has a subquadratic number of edges, then this trivial upper bound is improved.

\begin{theorem}\label{thm:sparse}
    Let $F$ be an $r$-partite $r$-graph with $e(F) = c v(F)$, where $c$ is a positive real number. Then we have
    \begin{equation*}
        s(F) \leq 2 c \cdot v(F)^{r-1}.
    \end{equation*}
\end{theorem}

Indeed, one can improve \Cref{thm:sparse} if $F$ is a very unbalanced $r$-graph, but for convenience, we state \Cref{thm:sparse} as above.

Our method provides a general approach to bound Sidorenko exponents for a broad family of hypergraphs, especially when their link profiles consist of sparse hypergraphs. To apply \Cref{thm:main-domination}, we should find a sufficiently sparse dominating hypergraph that contains all link hypergraphs of the given hypergraph. Unfortunately, there is a lot of uncertainty about dominating hypergraphs. Hence, it seems hard to find a clean statement that covers a large sensible family of hypergraphs aside from $3$-uniform tight cycles. Instead, we illustrate one example that shows how our mechanism can be applied to various hypergraphs.

For an integer $k \geq 2$, we say a graph is $(k\times k)$-grid if it is obtained by a cartesian product of two paths of $k$ vertices.

\begin{theorem}\label{thm:grid}
    Let $F$ be a $3$-partite $3$-graph with link profile $(G_1, \dots, G_t)$ such that $\bigcup_{i\in [t]} G_i$ is a subgraph of $(k \times k)$-grid. Then we have
    \begin{equation*}
        s(F) \leq 8 e(F) + 8 k^2.
    \end{equation*}
\end{theorem}

We note that for all $k \geq 3$, a $(k \times k)$-grid is not a dominating hypergraph due to \cite[Proposition 2.3]{dominating-graphs}. In fact, the proof of \Cref{thm:grid} starts with finding a small and sparse dominating graph that contains $(k \times k)$-grid as a subgraph (see \Cref{sec:proof-applications}). In this context, \Cref{thm:grid} and its proof suggest a way to apply our main theorems together with finding dominating hypergraphs by using several methods to construct dominating hypergraphs, which will be discussed in \Cref{subsec:norming-dominating,sec:proof-applications}. 

We summarize this section with a quick application of \Cref{thm:main-extremal}.

\begin{theorem}\label{thm:bipartite-links}
    Let $F$ be a $3$-graph on the vertex partition $V_1 \cup V_2 \cup V_3$ and $G$ be a bipartite graph. Assume for all $v\in V_3$, the link graph $\calL_F(v)$ is a subgraph of $G$. Then the following holds.
    \begin{equation*}
        \mathrm{ex}(n, F) \leq O_F\left( n^{3 - \frac{1}{e(G) + v(G)/2}} \right).
    \end{equation*}
\end{theorem}

\begin{proof}[Proof of \Cref{thm:bipartite-links}]
    Observe that $F$ is a sub-hypergraph of $G(v(F))$. By \Cref{thm:main-extremal}, we deduce the following.
    \begin{equation}\label{eq:bipartite-links}
        \mathrm{ex}(n, G(v(F))) \leq O_F\left(n^{3 - \frac{1}{s(G)}}\right).
    \end{equation}
    As a direct corollary of \Cref{thm:Conlon-Fox-Sudakov},  the Sidorenko exponent of $G$ is bounded above by $e(G) + \frac{v(G)}{2}$. Hence, together with \eqref{eq:bipartite-links}, we obtain
    \begin{equation*}
        \mathrm{ex}(n, F) \leq \mathrm{ex}(n, G(v(F))) \leq O_F\left( n^{3 - \frac{1}{e(G) + v(G)/2}} \right).
    \end{equation*}
    This completes the proof.
\end{proof}


\section{Preliminaries}\label{sec:prelim}

\subsection{Notations}\label{subsec:notations}

In this paper, we use standard notations in hypergraph theory. Let $r \geq 1$ be a positive integer. For an $r$-graph $F$, we denote by $V(F)$ and $E(F)$ the set of vertices and edges of $F$, respectively. We also write their sizes as $v(F)$ and $e(F)$. As we already defined in \Cref{subsec:results}, for a vertex subset $U\subseteq V(F)$, we denote $d_F(U)$ the number of edges of $F$ that contain at least one vertex of $U$. If $U$ is a set of single vertex $\{v\}$, then we simply write it as $d_F(v)$, which is the same as the degree of $v$ in $F$. We denote by $\Delta(F)$ and $\delta(F)$ the maximum and minimum degree of $F$, respectively. For a vertex subset $U\subseteq V(F)$, we write $F - U$ as a hypergraph obtained from $F$ by removing all vertices in $U$.

Let $F$ be an $r$-graph and $F'$ be a sub-hypergraph of $F$. For a homomorphism $\phi\in \Hm(F, H)$, we denote by $\phi\vert_{F'}$ the restriction of $\phi$ on $F'$. That is, $\phi\vert_{F'}$ is a homomorphism of $F'$ in $H$ such that for all $e\in E(F')$, $\phi\vert_{F'}(e) = \phi(e)$.
 
Let $r \geq 2$, $F$ be an $r$-graph, and $S$ be an $(r-1)$-graph. Then we denote $N_F(S)$ as the set of vertices of $V(F)$ such that every edge of $S$ forms an edge in $H$ together with those vertices. For each vertex $v\in V(F)$, we define \emph{downward hypergraph} $\calD_F(v)$ as an $(r-1)$-graph, which is similar to the link hypergraph, such that
$$
    V(\calD_F(v)) \defeq V(F)
$$
and
$$
    E(\calD_F(v)) \defeq \{f\in V(\calD_F(v))^{(r-1)}: f\cup \{v\} \in E(F)\}.
$$ 
Note that the only difference between downward hypergraphs and link hypergraphs is their vertex sets. We define those hypergraphs independently to pursue notational convenience in the rest of the paper.


\subsection{Tensor power trick}\label{subsec:tensor}
There is a technique, so-called \emph{tensor power trick}, which is a very useful argument when proving that a weak version of a certain homomorphism inequality indeed implies a stronger inequality. This technique was used in the elementary proof that trees are Sidorenko (this implies the well-known Blakley-Roy inequality~\cite{Blakley-Roy}) by Alon and Ruzsa~\cite{Alon-Ruzsa}. The tensor power trick has also been used in various areas. We refer to Tao's book~\cite{Tao} for a number of applications of this technique. This technique is now well-known, but for the completeness of the paper, we demonstrate the technique.

Let $F_1$ and $F_2$ are two $r$-graphs. We define \emph{tensor product} $F_1 \otimes F_2$ to be an $r$-graph on the vertex set $V(F_1) \times V(F_2)$ such that $((a_1, b_1), \dots, (a_r, b_r)) \in E(F_1 \otimes F_2)$ if and only if $(a_1, \dots, a_r)\in E(F_1)$ and $(b_1, \dots, b_r)\in E(F_2)$. For a positive integer $k > 0$, we define the $k$-fold tensor power $F^{\otimes k}$ by the $r$-graph $F \otimes \cdots \otimes F$, applying the tensor product $k$ times. Then we have the following observation, which is straightforward from the definitions.

\begin{observation}\label{obs:tensor}
    Let $F$ and $H$ be $r$-graphs and $k > 0$ be a positive integer. Then we have
    \begin{equation*}
        t_F(H^{\otimes k}) = t_F(H)^k.
    \end{equation*}
\end{observation}

By using \Cref{obs:tensor}, we obtain the following proposition, which will be used in the proof of \Cref{thm:main-sidorenko,thm:main-domination}.

\begin{proposition}\label{prop:tensor}
    Let $F$ be an $r$-graph and $s \geq 0$ be a real number. Assume there exists a positive constant $c > 0$ such that for all $r$-graph $H$, the inequality $t_F(H) \geq c \cdot  t_{K_r^{(r)}}(H)^s$ always hold. Then 
    \begin{equation*}
        s(F) \leq s.
    \end{equation*}
\end{proposition}

\begin{proof}[Proof of \Cref{prop:tensor}]
    If $c \geq 1$, then $s(F) \leq s$ directly follows from the definition of the Sidorenko exponent. Thus, we may assume that $c < 1$.
    
    Take a non-empty $r$-graph $H$ such that $t_F(H) < t_{K_r^{(r)}}(H)^s$. Let $c' \defeq t_F(H) / t_{K_r^{(r)}}(H)^s$. Note that $0 < c' < 1$. Then, for all positive integers $k > 0$, the following holds by \Cref{obs:tensor}.
    \begin{equation*}
        t_F(H^{\otimes k}) = t_F(H)^k = c'^k t_{K_r^{(r)}}(H)^{sk} = c'^k t_{K_r^{(r)}}(H^{\otimes k})^{s}.
    \end{equation*}
    Since $c' < 1$, for sufficiently large $k$, we deduce
    \begin{equation*}
        t_F(H^{\otimes k}) < c \cdot t_{K_r^{(r)}}(H^{\otimes k})^{s},
    \end{equation*}
    a contradiction. This implies that for all $r$-graph $H$, we have $t_F(H) \geq t_{K_r^{(r)}}(H)^s$. This completes the proof.
\end{proof}


\subsection{Norming and dominating hypergraphs}\label{subsec:norming-dominating}

In this section, we collect several techniques to prove that some hypergraphs are dominating by using a stronger notion, called \emph{weakly norming} hypergraphs.

Let $F$ be an $r$-graph and $f: [0, 1]^r \to \mathbb{R}$ be a bounded Lebesgue measurable function. Define functionals $\lVert \cdot \rVert_F$ and $\lVert \cdot \rVert_{w(F)}$ respectively as follows.
\begin{equation*}
    \lVert f \rVert_F \defeq \biggl\lvert \int \prod_{(v_1, \dots, v_r)\in E(F)} f(x_{v_1}, \dots, x_{v_r}) d\mu^{v(F)} \biggr\rvert^{\frac{1}{e(F)}}
\end{equation*}
and
\begin{equation*}
    \lVert f \rVert_{w(F)} \defeq \biggl( \int \prod_{(v_1, \dots, v_r)\in E(F)} \lvert f(x_{v_1}, \dots, x_{v_r}) \rvert d\mu^{v(F)} \biggr)^{\frac{1}{e(F)}},
\end{equation*}
where $\mu$ is an Lebesgue measure on $[0, 1]$.
If the functional $\lVert \cdot \rVert_F$ is a (semi-)norm, then we say $F$ is \emph{norming}, and if $\lVert \cdot \rVert_{w(F)}$ is a norm, then we say $F$ is \emph{weakly norming}. Note that if $F$ is norming, then so is weakly norming.

The concepts of norming and weakly norming hypergraphs are crucial in the area related to homomorphism counting since those norming properties guarantee various strong domination inequalities. For instance, if $F$ is a weakly norming, then it is also dominating, so is Sidorenko. Indeed, Lov\'{a}sz~\cite{Lovasz-limit} raised a problem to characterize (weakly)norming graphs.

A quick demonstration of the fact that weakly norming hypergraphs are also dominating hypergraphs follows. To prove $\lVert \cdot \rVert_{w(F)}$ is a norm, all the requirements are straightforward from the definition except the triangle inequality. In order to prove the triangle inequality, Hatami~\cite{Hatami,Hatami-thesis} proved that it is equivalent to showing the following Cauchy--Schwartz--Gowers inequality. Let $\chi: E(F) \to [e(F)]$ be an edge coloring, not necessarily proper and let $\calF = (f_1, \dots, f_{e(F)})$ be a family of bounded Lebesgue measurable functions on $[0, 1]^r$. Define
$$
    \langle\calF; \chi\rangle_F \defeq \int \prod_{e = (v_1, \dots, v_r)\in E(F)} f_{\chi(e)}(x_{v_1}, \dots, x_{v_r}) d\mu^{v(F)}.
$$
Then showing the triangle inequality is equivalent to proving
\begin{equation}\label{eq:CSG}
    \langle\calF; \chi\rangle_F \leq \prod_{e\in E(F)} \lVert f_{\chi(e)} \rVert_{w(F)}
\end{equation}
hold for all choices of $\calF$ and $\chi$.\footnote{Hatami~\cite{Hatami-thesis} proved this equivalence only for graph cases, but as he mentioned in the same thesis, its hypergraph generalization is straightforward from his argument.}

From \eqref{eq:CSG}, we can deduce that weakly norming implies dominating. Assume $F$ is a weakly norming $r$-graph and consider an arbitrary sub-hypergraph $F'$ of $F$. Let $H$ be an $r$-graph on $n$ vertices $(v_1, \dots, v_n)$. Consider a function $h: [0, 1]^r \to \mathbf{R}$ as
$$
    h(x_1, \dots, x_r) \defeq 
    \begin{cases}
        1 &\text{if $(v_{\lfloor nx_1 \rfloor}, \dots, v_{\lfloor nx_r \rfloor}) \in E(H)$},\\
        0 &\text{otherwise}.
    \end{cases}
$$
Observe that $t_F(H) = \lVert h \rVert_{w(F)}^{e(F)}$ and $t_{F'}(H) = \lVert h \rVert_{w(F')}^{e(F')}$. Setting $f_{\chi(e)} = h$ for all $e\in E(F')$ and $f_{\chi(e)} = 1$ for all $e\notin E(F')$ in \eqref{eq:CSG}. Then we have the inequality $t_F(H)^{\frac{1}{e(F)}} \geq t_{F'}(H)^{\frac{1}{e(F')}}$, hence $F$ is dominating.

There are several hypergraph classes that are known to be (weakly)norming. For instance, $r$-graph $K^{(r)}_{2, \dots, 2}$ is a norming hypergraph, now it is known as Gowers' octahedral norm, which was used in his work on hypergraph regularity~\cite{Gowers-3graph,Gowers-hypergraph} and is now a widely used tool for measuring quasirandomness in extremal combinatorics. 
A systematic research on (weakly)norming graphs and hypergraphs was initiated by Hatami~\cite{Hatami}. He showed that even cycles and complete $r$-partite $r$-graphs are norming, and $n$-dimensional hypercube graphs are weakly norming. In addition, Lov\'{a}sz~\cite{Lovasz-limit} proved that the complete graph minus a perfect matching is also weakly norming. Generalising these results, Conlon and Lee~\cite{Finite-reflection} found a large class of (weakly)norming hypergraphs and introduced a general method to verify (weakly)norming hypergraphs. Also, Conlon and Lee developed a method in~\cite{Finite-reflection} to find broader classes of dominating graphs~\cite{dominating-graphs}, which are not weakly norming graphs.

We close this section with the following two propositions. These propositions provide a way to construct new weakly norming and dominating hypergraphs recursively. 

\begin{proposition}[Hatami~\cite{Hatami}]\label{prop:weakly-tensor}
    Let $F_1$ and $F_2$ be two weakly norming $r$-graphs. Then their tensor product $F_1 \otimes F_2$ is also a weakly norming $r$-graph.
\end{proposition}

\begin{proposition}[Conlon--Lee~\cite{dominating-graphs}]\label{prop:dominating-tensor}
    Let $G$ be a dominating graph. Then for any integer $m \geq 1$, the graph $H\otimes K^{(2)}_{m, m}$ is a dominating graph.
\end{proposition}


\section{Proof of main theorems}\label{sec:proof}
In this section, we prove \Cref{thm:main-sidorenko,thm:main-domination,thm:main-extremal}.
The proofs of \Cref{thm:main-sidorenko,thm:main-domination} are appropriate hypergraph generalizations of the proof of \Cref{thm:Conlon-Fox-Sudakov}. In order to prove \cref{thm:main-sidorenko,thm:main-domination}, we need to collect the following simple lemmas.

\begin{lemma}\label{lem:sidorenko-dominating}
    Let $M$ be an $r$-graph, which is a disjoint union of Sidorenko $r$-graphs $S_1, \dots, S_{\ell}$. Let $S'$ be a disjoint union of some members in $\{S_1, \dots, S_t\}$. 
    Then for every $n$-vertex $r$-graph $H$, we have the following inequality.
    \begin{equation*}
        \frac{\hm(M, H)}{\hm(M - V(M'), H)} \geq \left(t_{K_{r}^{(r)}}(H)\right)^{d_M(V(M'))} n^{v(M')}.
    \end{equation*}
\end{lemma}

\begin{proof}[Proof of \Cref{lem:sidorenko-dominating}]
    Since $M'$ is a disjoint union of connected components of $M$, we have 
    $$
        \hm(M, H) = \hm(M', H) \cdot \hm(M - V(M'), H),
    $$ so the following holds.
    \begin{equation}\label{eq:sidorenko-dominating}
        \frac{\hm(M, H)}{\hm(M - V(M'), H)} = \hm(M', H) = t_{M'}(H) n^{v(M')}.
    \end{equation} 
    Since $M'$ is a disjoint union of Sidorenko $(r-1)$-graphs, so is Sidorenko and $d_M(M') = e(M')$. Thus, we have $t_{M'}(H) \geq \left(t_{K_{r}^{(r)}}(H)\right)^{d_M(M')}$. Together with \eqref{eq:sidorenko-dominating}, we obtain the desired inequality. This completes the proof. 
\end{proof}

We also have a similar lemma for dominating hypergraphs.

\begin{lemma}\label{lem:dominating-dominating}
    Let $M$ be a dominating $r$-graph $M'$ be a sub-hypergraph of $M$.
    Then for every $n$-vertex $r$-graph $H$, we have the following inequality.
    \begin{equation*}
        \frac{\hm(M, H)}{\hm(M - V(M'), H)} \geq \left(t_{K_{r}^{(r)}}(H)\right)^{d_M(V(M'))} n^{v(M')}.
    \end{equation*}
\end{lemma}

\begin{proof}[Proof of \Cref{lem:dominating-dominating}]
    Let $L = M - V(M')$.
    Observe that $e(L) = e(M) - d_M(V(M'))$. By the definition of the dominating hypergraph, we have
    $$
        t_M(H)^{\frac{1}{e(M)}} \geq t_{L}(H)^{\frac{1}{e(M) - d_M(V(M'))}},
    $$ so the following inequality holds.
    \begin{equation}\label{eq:dominating-dominating}
         t_M(H)^{\frac{e(M) - d_M(V(M'))}{e(M)}} \geq t_{L}(H).
    \end{equation}
    From \eqref{eq:dominating-dominating}, we deduce the following.
    \begin{equation}\label{eq:dominating2}
        \frac{\hm(M, H)}{\hm(L, H)} = \frac{t_M(H)}{t_{L}(H)} n^{v(M) - v(L)} \geq t_M(H)^{\frac{d_M(V(M'))}{e(M)}} n^{v(M')}.
    \end{equation}
    Since $M$ is dominating, so is Sidorenko, hence we have
    \begin{equation}\label{eq:dominating3}
        t_M(H) \geq t_{K_r^{(r)}}(H)^{e(M)}.
    \end{equation}
    By combining \eqref{eq:dominating2} and \eqref{eq:dominating3}, we obtain our desired inequality as follows.
    \begin{align*}
        \frac{\hm(M, H)}{\hm(M - V(M'), H)} &= \frac{\hm(M, H)}{\hm(L, H)}\\ &\geq t_M(H)^{\frac{d_M(V(M'))}{e(M)}} n^{v(M')}\\ &\geq \left(t_{K_{r}^{(r)}}(H)\right)^{d_M(V(M'))} n^{v(M')}. 
    \end{align*}
    This completes the proof.
\end{proof}

We are now ready to prove \Cref{thm:main-sidorenko,thm:main-domination}.


\subsection{Proof of \Cref{thm:main-sidorenko,thm:main-domination}}\label{subsec:proof-main12}

In this section, we prove \Cref{thm:main-sidorenko,thm:main-domination} simultaneously. To do this, observe the following.

\begin{observation}\label{obs:disjoint}
    Let $F$ be an $r$-partite $r$-graph and $M$ be an $(r-1)$-graph. Assume that the link profile $(\calL_F(v_1), \dots, \calL_F(v_t))$ of $F$ has the following property:
    \begin{enumerate}
        \item[$(1)$] $\calL_F(v_1) = M$.
        \item[$(2)$] For all $i\in [t]$, the link hypergraph $\calL_F(v_i)$ is a disjoint unions of some connected components of $M$.
    \end{enumerate}
    Then we have
    $$
        e(F) = \sum_{i\in [t]} d_F(V(\calL_F(v_i))). 
    $$
\end{observation}
The proof of \Cref{obs:disjoint} is almost trivial, so we omit it. Due to \Cref{obs:disjoint}, the following theorem implies both \Cref{thm:main-sidorenko,thm:main-domination}.

\begin{theorem}\label{thm:unified}
    Let $r \geq 2$ be an integer and $M$ be an $(r-1)$-graph.
    Let $F$ be an $r$-graph with link profile $(\calL_F(v_1), \dots, \calL_F(v_t))$ such that $\calL_F(v_1) = M$ and for all $i \in [t]$, the link hypergraph $\calL_F(v_i)$ is a sub-hypergraph of $M$. Assume that $F$ and $M$ satisfy at least one of the following:
    \begin{enumerate}
        \item[$(1)$] $M$ is a disjoint unions of Sidorenko $(r-1)$-graphs, say $S_1, \dots, S_{\ell}$, and for each $i\in [t]$, the link hypergraph $\calL_F(v_i)$ is a disjoint unions of some members in $\{S_1,\dots, S_{\ell}\}$. 
        \item[$(2)$] $M$ is a dominating $(r-1)$-graph. 
    \end{enumerate}
    Then we have the following inequality.
    \begin{equation*}
        s(F) \leq \sum_{i\in [t]} d_M(V(\calL_F(v_i))).
    \end{equation*}
\end{theorem}

\begin{proof}[Proof of \Cref{thm:unified}]
    Let $r$, $M$, and $F$ be the integer and hypergraphs that are stated in \Cref{thm:unified}. Let $H$ be an $n$-vertex $r$-graph. Our goal is to show that there is a positive real number $c > 0$ that depends only on $F$ such that the inequality
    \begin{equation}\label{eq:weak}
        \hm(F, H) = t_F(H)n^{v(F)} \geq c \left(t_{K_r^{(r)}}(H)\right)^{\sum_{i\in [t]} d_M(V(\calL_F(v_i)))} n^{v(F)}
    \end{equation} holds for all choices of $H$. Indeed, by \Cref{prop:tensor}, to finish the proof, it suffices to prove the inequality \eqref{eq:weak} holds. 

    Let $(\calL_F(v_1), \dots, \calL_F(v_t)) = (M, M_2, \dots, M_t)$ be the link profile of $F$, where $M_1 = M$.    
    For each $2 \leq i \leq t$, we say a homomorphism $\psi\in \Hm(M_i, K_n^{(r-1)})$ is \emph{rare} with respect to $i$ if 
    $$
        |N_H(\psi(M_i))| \leq (2t)^{-e(M)}\cdot \left(t_{K_r^{(r)}}(H)\right)^{d_M(V(M_i))} n.
    $$ 
    For a vertex $v\in V(H)$, we say $v$ is \emph{bad} with respect to $i$ if
    $$
        \lvert \{\phi\in \Hm(M, \calD_{H}(v)): \phi\vert_{M_i} \text{ is rare} \} \rvert \geq \frac{1}{2t} \cdot \hm(M, \calD_H(v)).
    $$
     Lastly, a vertex $v\in V(H)$ is \emph{good} if for all $2 \leq i \leq t$, $v$ is not a bad vertex with respect to $i$.

    \begin{claim}\label{clm:good-vertices}
        $\sum_{(v: \text{good})} t_{K_{r-1}^{(r-1)}}(\calD_H(v)) \geq \frac{1}{2}\cdot t_{K_r^{(r)}}(H) n.$
    \end{claim}

    \begin{claimproof}[Proof of \Cref{clm:good-vertices}]
        For each $u\in V(H)$ and $2 \leq i \leq t$, let $Z_u^{(i)}$ be the number of $\psi\in \Hm(M_i, K_n^{(r-1)})$ such that $\psi$ is rare with respect to $i$ and $u \in N_H(\psi(M_i))$. Denote $Z^{(i)}$ as a number $\sum_{u\in V(H)} Z_u^{(i)}$
        Then we have 
        \begin{equation}\label{eq:goodvertices1}
            Z^{(i)} = \sum_{(\psi: \text{rare w.r.t. $i$})} N_H(\psi(M_i)) \leq (2t)^{-e(M)}\cdot \left(t_{K_r^{(r)}}(H)\right)^{d_M(V(M_i))} n^{v(M_i) + 1}.
        \end{equation}
        On the other hand, we can lower bound $Z^{(i)}$ by restricting the vertex $u$ to be bad with respect to $i$.
        Assume $u$ is a bad vertex with respect to $i$.  Then, by the definition of bad, the number of homomorphisms $\phi\in \Hm(M, \calD_H(u))$ such that $\phi\vert_{M_i}$ is rare with respect to $i$ is at least $\frac{1}{2t}\cdot \hm(M, \calD_H(u))$. Observe that for every fixed $\psi\in \Hm(M_i, \calD_H(u))$, the number of choices of $\phi \in \Hm(M, \calD_H(u))$ such that $\phi\vert_{M_i} = \psi$ is at most $\hm(M - M_i, \calD_H(u))$. Hence, we have the following.

        \allowdisplaybreaks
        \begin{align*}
            \hm(M - M_i, \calD_H(u)) \cdot Z_u^{(i)} &\geq
            |\{\phi \in \Hm(M, \calD_H(u)) : \phi\vert_{M_i} \text{ is rare w.r.t. $i$} \}|\\ &\geq \frac{1}{2t}\cdot \hm(M, \calD_H(u)).
        \end{align*}
        Thus, we have the following inequality.
        \begin{equation}\label{eq:goodvertices2}
            Z_u^{(i)} \geq \frac{1}{2t}\cdot \frac{\hm(M, \calD_H(u))}{\hm(M - M_i, \calD_H(u))}
        \end{equation}
        By applying \Cref{lem:sidorenko-dominating,lem:dominating-dominating} to \eqref{eq:goodvertices2}, we have
        \allowdisplaybreaks
        \begin{align}
            Z^{(i)} &\geq \sum_{(u: \text{bad w.r.t. $i$})} Z_u^{(i)}\nonumber\\
            &\geq \sum_{(u: \text{bad w.r.t. $i$})} \frac{1}{2t}\cdot \frac{\hm(M, \calD_H(u))}{\hm(M - M_i, \calD_H(u))}\nonumber\\
            &\geq \frac{1}{2t} \cdot \sum_{(u: \text{bad w.r.t. $i$)}} \left( t_{K_{r-1}^{(r-1)}}(\calD_H(u)) \right)^{d_M(V(M_i))} n^{v(M_i)}\nonumber\\
            &\geq \frac{1}{2t} \cdot \left(\frac{\sum_{(u: \text{bad w.r.t. $i$})} t_{K_{r-1}^{(r-1)}}(\calD_H(u))}{n}\right)^{d_M(V(M_i))} n^{v(M_i) + 1}.\label{eq:goodvertices3}
        \end{align}
        The last inequality \eqref{eq:goodvertices3} follows from the fact that the number of bad vertices with respect to $i$ is at most $n$ and the convexity of the function $f(x) = x^{d_M(V(M_i))}$.
        From \eqref{eq:goodvertices1} and \eqref{eq:goodvertices3}, we obtain the following inequality.

        \begin{equation}\label{eq:goodvertices4}
            \sum_{(u: \text{bad w.r.t. $i$})} t_{K_{r-1}^{(r-1)}}(\calD_H(u)) \leq \frac{1}{2t} \cdot t_{K_r^{(r-1)}}(H) n.
        \end{equation}
        A simple double counting yields that
        $$
            \sum_{v\in V(H)} t_{K_{r-1}^{(r-1)}}(d_H(v))n^{r-1} = t_{K_r^{(r)}}(H)n^r.
        $$
        Hence, by \eqref{eq:goodvertices4}, we have the following inequalities.

        \allowdisplaybreaks
        \begin{align*}
            \sum_{(v: \text{good})} t_{K_{r-1}^{(r-1)}}(\calD_H(v)) n^{r-1} &= \sum_{v} t_{K_{r-1}^{(r-1)}}(\calD_H(v)) n^{r-1} - \sum_{2 \leq i \leq t} \sum_{(u: \text{bad w.r.t. $i$})} t_{K_{r-1}^{(r-1)}}(\calD_H(u)) n^{r-1}\\
            &= t_{K_r^{(r)}}(H)n^r - \sum_{2 \leq i \leq t} \sum_{(u: \text{bad w.r.t. $i$})} t_{K_{r-1}^{(r-1)}}(\calD_H(u)) n^{r-1}\\
            &\geq t_{K_r^{(r)}}(H)n^r - \sum_{2 \leq i \leq t} \frac{1}{2t}\cdot t_{K_r^{(r)}}(H) n^r\\
            &\geq \frac{1}{2} \cdot t_{K_r^{(r)}}(H)n^r.
        \end{align*}
        This completes the proof.
    \end{claimproof}

    Let $v \in V(H)$ be a vertex and $\phi \in \Hm(M, \calD_H(v))$. We say that $\phi$ is \emph{rich} with respect to $v$ if for all $2 \leq i \leq t$, the restriction $\phi\vert_{M_i}$ is not rare with respect to $i$.
    The next claim says that one can find many rich homomorphisms of $M$ in downward hypergraphs of a good vertex.

    \begin{claim}\label{clm:rich-homomorphisms}
        Let $u\in V(H)$ be a good vertex. Then there are at least $\frac{1}{2}\cdot \hm(M, \calD_H(u))$ rich homomorphisms with respect to $u$.
    \end{claim}

    \begin{claimproof}[Proof of \Cref{clm:rich-homomorphisms}]
        By the definition of a good vertex, for each $2 \leq i \leq t$, the number of homomorphisms $\phi\in \Hm(M, \calD_H(u))$ such that $\phi\vert_{M_i}$ is rare with respect to $i$ is at most $\frac{1}{2t} \cdot \hm(M, \calD_H(u))$. Thus, the number of rich homomorphisms is at least 
        $$
            \hm(M, \calD_H(u)) - \sum_{2 \leq i \leq t} \frac{1}{2t} \cdot \hm(M, \calD_H(u)) \geq \frac{1}{2} \cdot \hm(M, \calD_H(u)).
        $$
        This completes the proof.
    \end{claimproof}

    Now we show that \eqref{eq:weak} holds for some $c > 0$ that only depends on $F$. Our strategy is to pick a good vertex $v$ first and then find many rich homomorphisms in the downward hypergraph of $v$, which will play the role of $v_1$ in the $r$-th part of $V(F)$ such that $\calL_F(v_1) = M$. Then, by the definition of rich and rare homomorphisms, one can show that there are sufficiently many homomorphic copies of $F$ in $H$ by using \Cref{clm:good-vertices,clm:rich-homomorphisms}. From the definitions of good, rich, and rare, the following inequalities hold.

    \allowdisplaybreaks
    \begin{align}
        \hom(F, H) &\geq \sum_{(v: \text{good})} \sum_{(\phi\in \Hm(M, \calD_H(v)))} \prod_{2 \leq i \leq t} N_H(\phi\vert_{M_i})\nonumber\\
        &\geq \sum_{(v: \text{good})} \sum_{\substack{[\phi\in \Hm(M, \calD_H(v))\\: \text{rich w.r.t. $v$]}}} \prod_{2 \leq i \leq t} N_H(\phi\vert_{M_i})\nonumber\\
        &\geq \sum_{(v: \text{good})} \sum_{\substack{[\phi\in \Hm(M, \calD_H(v))\\: \text{rich w.r.t. $v$]}}} \prod_{2\leq i \leq t}  \bigg[(2t)^{-e(M)} \cdot \left(t_{K_r^{(r)}} (H) \right)^{d_M(V(M_i))} n\bigg]\label{eq:use-rich}\\
        &= (2t)^{-t e(M)} \sum_{(v: \text{good})} \sum_{\substack{[\phi\in \Hm(M, \calD_H(v))\\: \text{rich w.r.t. $v$]}}} \left(t_{K_r^{(r)}} (H) \right)^{\sum_{2\leq i \leq t} d_M(V(M_i))} n^{t-1}\nonumber\\
        &\geq (2t)^{-t e(M)} n^{t-1} \sum_{(v: \text{good})} \frac{1}{2} \cdot \hom(M, \calD_H(v)) \cdot \left(t_{K_r^{(r)}} (H) \right)^{\sum_{2\leq i \leq t} d_M(V(M_i))}\label{eq:use-many-rich}\\
        &= \frac{1}{2} \cdot (2t)^{-t e(M)}  \left(t_{K_r^{(r)}} (H) \right)^{\sum_{2\leq i \leq t} d_M(V(M_i))} n^{t-1} \sum_{(v: \text{good})} \hom(M, \calD_H(v))\nonumber\\
        &\geq \frac{1}{2} \cdot (2t)^{-t e(M)}  \left(t_{K_r^{(r)}} (H) \right)^{\sum_{2\leq i \leq t} d_M(V(M_i))} n^{v(M) + t - 1} \sum_{(v: \text{good})} \left(t_{K_{r-1}^{(r-1)}}(\calD_H(v)) \right)^{e(M)}\label{eq:use-sidorenko}\\
        &\geq \frac{1}{2} \cdot (2t)^{-t e(M)}  \left(t_{K_r^{(r)}} (H) \right)^{\sum_{2\leq i \leq t} d_M(V(M_i))} n^{v(M) + t}  \left(\frac{\sum_{(v: \text{good})} t_{K_{r-1}^{(r-1)}}(\calD_H(v))}{n}\right)^{e(M)}\label{eq:use-convexity}\\
        &\geq \frac{1}{4} \cdot (2t)^{-t e(M)}  \left(t_{K_r^{(r)}} (H) \right)^{\sum_{2\leq i \leq t} d_M(V(M_i))} n^{v(F)} \left(t_{K_r^{(r)}}(H)\right)^{e(M)}\label{eq:use-first-clm}\\
        &= \frac{1}{4} \cdot (2t)^{-t e(M)}  \left(t_{K_r^{(r)}} (H) \right)^{\sum_{i\in [t]} d_M(V(M_i))} n^{v(F)}.\nonumber
    \end{align}
    The inequality \eqref{eq:use-rich} holds by using the definition of rich homomorphisms, and \eqref{eq:use-many-rich} holds by \Cref{clm:rich-homomorphisms}. To deduce \eqref{eq:use-sidorenko}, we use the fact that $M$ is a Sidorenko $(r-1)$-graph. By using the convexity of a function $f(x) = x^{e(M)}$ and the fact that there are at most $n$ good vertices, we obtain the inequality \eqref{eq:use-convexity}. Lastly, from \Cref{clm:good-vertices}, the last inequality \eqref{eq:use-first-clm} holds.
     
    Hence, by taking $c = \frac{1}{4}\cdot (2t)^{-te(M)}$, the inequality \eqref{eq:weak} holds. This completes the proof.
\end{proof}


\subsection{Proof of \Cref{thm:main-extremal}}\label{subsec:proof-main3}
In this section, we prove our last main theorem, \Cref{thm:main-extremal}. Before we start, we define several notations. For two $r$-graphs $F$ and $H$, we say a homomorphism $\phi \in \Hm(F, H)$ is \emph{degenerated}, if $F$ and $\phi(F)$ is not isomorphic to each other. We call a homomorphism $\phi \in \Hm(F, H)$ \emph{proper} if it is not degenerated.

The first step to prove \Cref{thm:main-extremal} is showing that for a given $(r-1)$-graph $F$ and an integer $t$, the Sidorenko exponent of $F(t)$ is bounded above by $t$ times the Sidorenko exponent of $F$.

\begin{lemma}\label{lem:Sidorenko-lifting}
    Let $r \geq 2$ be an integer and $F$ be an $(r-1)$-graph. Then for all positive integers $t$, the following inequality holds.
    \begin{equation*}
        s(F(t)) \leq t \cdot s(F).
    \end{equation*}
\end{lemma}

\begin{proof}[Proof of \Cref{lem:Sidorenko-lifting}]
    Let $H$ be an $n$-vertex $r$-graph. From the definition of the $F(t)$, we have the following.

    \begin{equation}\label{eq:t-star}
        \hm(F(t), H) = \sum_{\phi\in \Hm(F, K_n^{(r-1)})} N_H(\phi(F))^t.
    \end{equation}
    To estimate the right-hand side of \eqref{eq:t-star}, observe that the following identity holds by double counting.

    \begin{equation}\label{eq:identity}
        \sum_{\phi\in \Hm(F, K_n^{(r-1)})} N_H(\phi(F)) = \sum_{v\in V(H)} \hm(F, \calD_H(v)) = \sum_{v\in V(H)} t_F(\calD_H(v)) n^{v(F)}.
    \end{equation}
    Then the following inequalities hold by \eqref{eq:t-star} and \eqref{eq:identity}.

    \allowdisplaybreaks
    \begin{align}
        \hm(F(t), H) &= \sum_{\phi\in \Hm(F, K_n^{(r-1)})} N_H(\phi(F))^t \nonumber\\
        &\geq \left(\frac{\sum_{\phi\in \Hm(F, K_n^{(r-1)})} N_H(\phi(F))}{n^{v(F)}}\right)^t n^{v(F)}\label{eq:t-star-convexity1}\\
        &= \left( \sum_{v\in V(H)} t_F(\calD_H(v)) \right)^t n^{v(F)}\nonumber\\
        &\geq \left[\sum_{v\in V(H)} \left(t_{K_{r-1}^{(r-1)}}(\calD_H(v))\right)^{s(F)} \right]^t n^{v(F)}\label{eq:sido-exp}\\
        &\geq \left[\left(\frac{\sum_{v\in V(H)} \left(t_{K_{r-1}^{(r-1)}}(\calD_H(v))\right)}{n}\right)^{s(F)} n\right]^t n^{v(F)}\label{eq:t-star-convexity2}\\
        &= \left(t_{K_r^{(r)}}(H) \right)^{t\cdot s(F)} n^{v(F) + t}\label{eq:t-star-final}
    \end{align}
    The inequalities \eqref{eq:t-star-convexity1} and \eqref{eq:t-star-convexity2} hold by the convexity of functions $f(x) = x^t$, $g(x) = x^{s(F)}$ and the fact that $\hm(F, K_n^{(r-1)}) = n^{v(F)}$, $v(H) = n$, respectively. From the definition of the Sidorenko exponent, we deduced \eqref{eq:sido-exp}.

    The last inequality \eqref{eq:t-star-final} implies
    \begin{equation*}
        t_{F(t)}(H) = \frac{\hm(F(t), H)}{n^{v(F(t))}} \geq \left( t_{K_r^{(r)}}(H) \right)^{t\cdot s(F)}.
    \end{equation*}
    This completes the proof.
\end{proof}

By applying \Cref{lem:Sidorenko-lifting} iteratively, to finish the proof, it suffices to prove \Cref{thm:main-extremal} for the case $\ell = 1$. 

\begin{lemma}\label{lem:main-extremal-one}
    Let $F$ be an $(r-1)$-graph and $t \geq 1$ be an integer. Then we have
    \begin{equation*}
        \mathrm{ex}(n, F(t)) \leq O_{F, t}(n^{r - \frac{1}{s(F)}}).
    \end{equation*}
\end{lemma}

\begin{proof}[Proof of \Cref{lem:main-extremal-one}]
    We will show that for every $n$-vertex $r$-graph with $e(H) \geq \frac{1}{r!}\cdot (t + 2v(F))^{\frac{1}{s(F)}} n^{r - \frac{1}{s(F)}}$ contains a isomomorphic copy of $F(t)$ as a sub-hypergraph. Observe that we have 
    $$
        t_{K_r^{(r)}}(H) \geq (t + 2v(F))n^{-\frac{1}{s(F)}}.
    $$
    
    By the definition of the Sidorenko exponent, for each $v\in V(H)$, we have the following inequality.

    \begin{equation*}
        \hom(F, \calD_H(v)) \geq \left( t_{K_{r-1}}^{(r-1)} (\calD_H(v)) \right)^{s(F)} n^{v(F)}.
    \end{equation*}
    From this, by double counting, we deduce the following inequalities.

    \allowdisplaybreaks
    \begin{align}
        \sum_{(\phi\in \Hm(F, K_n^{(r-1)}))} N_H(\phi(F)) &\geq \sum_{v\in V(H)} \left( t_{K_{r-1}}^{(r-1)} (\calD_H(v)) \right)^{s(F)} n^{v(F)}\nonumber\\
        &\geq \left( \frac{\sum_{v\in v(H)} t_{K_{r-1}}^{(r-1)} (\calD_H(v))}{n} \right)^{s(F)} n^{v(F) + 1}\label{eq:KST-convexity}\\
        &= \left(t_{K_r^{(r)}}(H)\right)^{s(F)} n^{v(F) + 1}\nonumber\\
        &\geq (t + 2v(F)) n^{v(F)}\nonumber\\
        &> (t + 2v(F) - 1) n^{v(F)}.\label{eq:KST-count}
    \end{align}
    The inequality \eqref{eq:KST-convexity} holds by the convexity of the function $f(x) = x^{s(F)}$.
    Note that the number of degenerated homomorphisms $\phi\in \Hm(F, K_n^{(r-1)})$ is at most 
    $$
        \sum_{i\in [v(F) - 1]} n^i \leq v(F)n^{v(F) - 1}.
    $$ Hence, we have the following,

    \begin{equation*}
        \sum_{\substack{[\phi\in \Hm(F, K_n^{(r-1)})\\: \text{ degenerated}]}} N_H(\phi(F)) \leq v(F)n^{v(F)}.
    \end{equation*} Toghther with \eqref{eq:KST-count}, we have
    \begin{equation}\label{eq:KST-proper}
        \sum_{\substack{[\phi\in \Hm(F, K_n^{(r-1)})\\: \text{ proper}]}} > (t + v(F) - 1)n^{v(F)}.
    \end{equation}
    Note that for every proper $\phi\in \Hm(F, K_n^{(r-1)})$, its image $\phi(F)$ is an isomorphic copy of $F$. Since there at most $n^{v(F)}$ proper homomorphisms of $F$ in $K_n^{(r-1)}$, by the pigeonhole principle and \eqref{eq:KST-proper}, there is a proper homomorphism $\phi\in \Hm(F, K_n^{(r-1)})$ such that $N_H(\phi(F)) \geq t + v(F)$. 
    This implies that there are at least $t$ vertices outside of the $\phi(F)$, which is isomorphic to $F$, such that together with all $(r-1)$-uniform edges of $F$ form edges of $H$. This means $H$ contains $F(t)$ as a sub-hypergraph, so 
    $$
        \mathrm{ex}(n, F(t)) < \frac{1}{r!}\cdot (t + 2v(F))^{\frac{1}{s(F)}} n^{r - \frac{1}{s(F)}}.
    $$ 
    This completes the proof.
\end{proof}

We are now ready to prove \Cref{thm:main-extremal}.

\begin{proof}[Proof of \Cref{thm:main-extremal}]
    Let $r > \ell \geq 1$ be integers and $F$ be an $(r - \ell)$-graph. By \Cref{lem:Sidorenko-lifting}, for all integers $t_1, \dots, t_{\ell} > 0$, we have the following.
    $$
        s(F(t_1, \dots, t_{\ell - 1})) \leq s(F) \prod_{i\in [\ell - 1]} t_i.
    $$ 
    Hence, \Cref{lem:main-extremal-one} implies
    $$
        \mathrm{ex}(n, F(t_1, \dots, t_{\ell})) = \mathrm{ex}(n, F(t_1, \dots, t_{\ell - 1})(t_{\ell})) \leq O_{F, t_1, \dots, t_{\ell}}
        \left(n^{r - \frac{1}{S(F) \prod_{i\in [\ell - 1]} t_i}} \right).
    $$
    This completes the proof.
\end{proof}


\section{Applying main theorems}\label{sec:proof-applications}
In this section, we prove \Cref{thm:tight-3-cycle,thm:sparse,thm:grid} by using \Cref{thm:main-domination}.
Our method to bound Sidorenko exponents is to find small and sparse dominating hypergraphs that contain each link hypergraph, and then apply \Cref{thm:main-extremal}. To finish the proof, we frequently use the following simple observation.

\begin{observation}\label{obs:subgraph}
    Let $F$ and $F'$ be $r$-graphs such that $F$ is a sub-hypergraph of $F'$. Then for all $r$-graphs $H$, the following inequality holds.
    \begin{equation*}
        t_{F}(H) \geq t_{F'}(H).
    \end{equation*}
    In particular, $s(F) \leq s(F')$.
\end{observation}

\subsection{Proof of \Cref{thm:tight-3-cycle}}\label{subsec:proof-tightcycle}

\begin{proof}[Proof of \Cref{thm:tight-3-cycle}]
    Let $F$ be a $3$-uniform tight cycle of length $3\ell$, where $\ell \geq 2$. Assume that $V(F) = [3\ell]$ and $E(F) = \{(i, i+1, i+2): i\in [3\ell]\}$ where addition is taken mod $3\ell$. For each $i \in \{0, 1, 2\}$, let $V_i$ be the set of numbers in $[3\ell]$ congruent with $i$ by modulo $3$. Observe that $F$ is a $3$-partite $3$-graph on vertex partition $V_0 \cup V_1 \cup V_2$. Observe that $G\defeq \bigcup_{v\in V_2} \calL_F(v)$ is a cycle of length $2\ell$ and each link graph $\calL_F(v)$ is a path of length $3$ contained in $G$. Now introduce a new vertex $u$ to make a new $3$-graph $F'$ as follows.
    $$
        V(F') \defeq V(F) \cup \{u\},
    $$
    and
    $$
        E(F') \defeq E(F) \cup \{e\cup \{u\}: e\in E(G)\}.
    $$
    Then the link profile of $F'$ is $(G, P_1, \dots, P_{\ell})$, where $G$ is isomorphic to an even cycle of length $2\ell$ and each $i\in [\ell]$, the graph $P_i$ is isomorphic to a path of length $3$, which is contained in $G$. Since even cycles are dominating as we discussed in \Cref{subsec:norming-dominating}, \Cref{thm:main-extremal} can be applied to $F'$. As even cycles are $2$-regular graphs, we have the following by \Cref{thm:main-extremal}
    \begin{equation*}
        s(F') \leq 2\ell + \sum_{i\in [\ell]} 5 = 7\ell.
    \end{equation*}
    As $F$ is a sub-hypergraph of $F'$, we deduce the following from \Cref{obs:subgraph} and the above inequality.
    \begin{equation*}
        s(F) \leq s(F') \leq 7\ell.
    \end{equation*}
    This completes the proof.
\end{proof}

\subsection{Proof of \Cref{thm:sparse}}\label{subsec:proof-sparse}

The proof of \Cref{thm:sparse} is almost the same as the proof of \Cref{thm:tight-3-cycle}. The only essential difference is the following. By the structure of $3$-partite $3$-uniform tight cycles, we can bound their Sidorenko exponents without paying much cost. However, in general, we do not know the existence of sufficiently small and sparse dominating hypergraphs that contain all link hypergraphs of a given hypergraph. Thus, to prove \Cref{thm:sparse}, we use the fact that complete $r$-partite $r$-graphs are dominating hypergraphs.

\begin{proof}[Proof of \Cref{thm:sparse}]
    Let $F$ be an $r$-partite $r$-graph on vertex partition $V_1 \cup \cdots \cup V_r$, where $|V_i| = t_i$ for each $i\in [r]$. We may assume that $F$ has no isolated vertices. Let $t\defeq \sum_{i\in [r]} t_i$ and assume $e(F) = c t$ for some positive real number $c > 0$. Without loss of generality, assume that $t_1 \leq \cdots, \leq t_r$. Since $F$ does not have an isolated vertex, $c \geq \frac{t_r}{t} \geq \frac{1}{r}$.
    
    Let $(L_1, \dots, L_{t_r})$ be the link profile of $F$. Note that for each $i\in [t_r]$, the link hypergraph $L_i$ is a sub-hypergraph of $K^{(r-1)}_{t_1, \dots, t_{r-1}}$. Define a new hypergraph $F'$ by introducing a new vertex $u$ as follows.
    $$
        V(F') \defeq V(F) \cup \{u\},
    $$
    and
    $$
        E(F') \defeq E(F) \cup \{e\cup \{u\}: e\in \prod_{i\in [r-1]} V_i\}.
    $$
    Since every complete $(r-1)$-partite $(r-1)$-graphs are dominating as discussed in \Cref{subsec:norming-dominating}, \Cref{thm:main-domination} can be applied to $F'$. By \Cref{thm:main-domination}, the following holds.
    \begin{equation}\label{eq:sparse}
        s(F') \leq \prod_{i\in [r-1]} t_i + \sum_{j\in [t_r]} d_{K^{(r-1)}_{t_1, \dots, t_{r-1}}}\left( V(L_i) \right).
    \end{equation}
    Observe that for each $j\in [r-1]$ and $v\in V_j$, the degree of $v$ in $K^{(r-1)}_{t_1, \dots, t_{r-1}}$ is $\prod_{i\in [r-1]\setminus \{j\}} t_i$.
    Also observe that for each $i\in [t_r]$ and $j\in [r-1]$, the inequaliy $\lvert V_i \cap V(L_i) \rvert \leq e(L_i)$ holds. Hence, we obtain the following inequality for each $i\in [t_r]$.
    \allowdisplaybreaks
    \begin{align}
        d_{K^{(r-1)}_{t_1, \dots, t_{r-1}}}\left( V(L_i) \right) &\leq e(L_i) \sum_{j\in [r-1]} \left(\prod_{k\in [r-1]\setminus \{j\}} t_k\right).\nonumber \\
        &\leq e(L_i) \sum_{j\in [r-1]} t_j^{r-2}\label{eq:Muirhead}\\
        &\leq e(L_i) t^{r-2}.\nonumber
    \end{align}
    The inequality \eqref{eq:Muirhead} holds by Muirhead's inequality.

    Note that $e(F) = \sum_{i\in [t_r]} e(L_i) = ct$. Hence, \eqref{eq:sparse} implies
    \allowdisplaybreaks
    \begin{align*}
        s(F') &\leq \prod_{i\in [r-1]} t_i + \sum_{j\in [t_r]} d_{K^{(r-1)}_{t_1, \dots, t_{r-1}}}\left( V(L_i) \right)\\
        &\leq \left(\frac{t}{r}\right)^{r-1} + \sum_{i\in [t_r]} e(L_i)t^{r-1}\\
        &\leq \left(\frac{t}{r}\right)^{r-1} + ct^{r-1}\\
        &\leq 2c t^{r-1}.
    \end{align*}
    The last inequality holds since $c \geq \frac{1}{r}$.

    As $F$ is a sub-hypergraph of $F$, by \Cref{obs:subgraph},
    \begin{equation*}
        s(F) \leq s(F') \leq 2c t^{r-1}.
    \end{equation*}
    This completes the proof.
\end{proof}

\subsection{Proof of \Cref{thm:grid}}\label{subsec:proof-grid}
For an integer $\ell \geq 3$, denote $C_{\ell}$ by the cycle of length $\ell$. To prove \Cref{thm:grid}, we need the following lemma. Denote $G_k$ and $T_k$ by the $(k \times k)$-grid and $C_{2k} \otimes C_{2k}$, respectively.

\begin{lemma}\label{lem:cycle-tensor}
    For any integer $k \geq 2$, the graph $T_k$ satisfies the following.
    \begin{enumerate}
        \item[$(a)$] $T_k$ is a $4$-regular graph on $4k^2$ vertices.
        \item[$(b)$] $T_k$ is a dominating graph. 
        \item[$(c)$] $T_k$ contains $G_k$ as a subgraph. 
    \end{enumerate}
\end{lemma}

\begin{proof}[Proof of \Cref{lem:cycle-tensor}]
    The statement $(a)$ directly follows from the definition of the tensor product, and $(b)$ follows from \Cref{prop:weakly-tensor}. Thus, it suffices to prove $(c)$ to complete the proof.

    From the definition of the tensor product, the graph $T_k$ has vertex and edge sets as follows.
    $$
        V(T_k) \defeq [2k] \times [2k]
    $$
    and
    $$
        E(T_k) \defeq \{\{(i,j), (i', j')\}\in \binom{V(T_k)}{2}: \lvert i - i' \rvert \equiv \lvert j - j' \rvert \equiv 1 \pmod{2k}\}\}.
    $$

    Let $U \defeq \{(k + i - j - 1, i + j - 1): i, j\in [k]\}$ be a vertex subset of $T_k$. Then it is straightforward that the subgraph of $T_k$ induced by the vertex subset $U$ is isomorphic to $G_k$. This completes the proof.
\end{proof}

With \Cref{lem:cycle-tensor} in mind, the proof of \Cref{thm:grid} is almost the same as that of \Cref{thm:tight-3-cycle,thm:sparse}.

\begin{proof}[Proof of \Cref{thm:grid}]
    Let $F$ be a $3$-partite $3$-graph with link profile $(L_1, \dots, L_t)$ such that $\bigcup_{i\in [t]} L_i$ is a subgraph of $G_k$. Define a $3$-graph $F'$ as follows.
    $$
        V(F') \defeq V(T_k) \cup \{u\}, 
    $$ where $u$ is disjoint from $V(T_k)$ and
    $$
        E(F') \defeq E(F) \cup \{e\cup \{u\}: e\in E(T_k)\}.
    $$ 
    
    Since $T_k$ is a bipartite graph, the new graph $F'$ is a $3$-partite $3$-graph with link profile $(T_k, L_1, \dots, L_t)$. By $(c)$ of \Cref{lem:cycle-tensor}, $G_k$ is a subgraph of $T_k$, the link graph $L_i$ is a subgraph of $T_k$ for all $i\in [k]$. Also, by $(b)$ of \Cref{lem:cycle-tensor}, $T_k$ is a dominating graph. Thus, \Cref{thm:main-domination} can be applied to $F'$.

    By $(a)$ of \Cref{lem:cycle-tensor}, we obtain the following.
    \begin{equation*}
        s(F') \leq 8k^2 + \sum_{i\in [t]} d_M(L_i) \leq 8k^2 + \sum_{i\in [t]} 8 e(L_i) = 8k^2 + 8e(F).
    \end{equation*}
    Since $F$ is a sub-hypergraph of $F'$, we have $s(F) \leq 8 k^2 + 8e(F)$ by \Cref{obs:subgraph}. This completes the proof.
    
\end{proof}


\section{Concluding remarks}\label{sec:concluding}
In this article, we found a new broad class of Sidorenko hypergraphs and also suggested a framework for bounding Sidorenko exponents of various hypergraphs. In particular, as an application of our framework, \Cref{thm:tight-3-cycle} provides an upper bound on the Sidorenko exponent of a $3$-uniform tight cycle, which is tight up to a multiplicative constant. We believe that $3$-partite $3$-uniform tight cycles are indeed nearly Sidorenko. Hence, we raise the following conjecture.

\begin{conjecture}\label{conj:3-tight-cycle}
    \begin{equation*}
        s\left(C^{(3)}_{3\ell} \right)  = (3 + o_{\ell}(1))\ell.
    \end{equation*}
\end{conjecture}

For the higher uniformity cases, we could not find reasonable upper bounds on Sidorenko exponents of tight cycles. However, we conjecture that there would be an analogous upper bound as \Cref{thm:tight-3-cycle} for $r$-partite $r$-uniform tight cycles for all $r \geq 4$. 

\begin{conjecture}\label{conj:k-tight-cycle}
    For all $r \geq 4$, there is a constant $K_r$ only depending on $r$ such that the following holds for all $\ell \geq 2$.
    \begin{equation*}
        s\left(C^{(r)}_{r\ell} \right) \leq K_r r\ell.
    \end{equation*}
\end{conjecture}
Note that the current best upper bound is $s\left( C^{(r)}_{r\ell} \right) \leq 2r^{r-1} \ell^{r-1}$, obtained by applying \Cref{thm:sparse}.

Inspired by \Cref{thm:sparse}, it would be very interesting to find a linear upper bound for Sidorenko exponents of sparse hypergraphs, especially hypergraphs with bounded maximum degree. We do not have strong confidence in this, so we do not make this a conjecture, but instead, we raise a problem as follows.

\begin{problem}\label{prob:sparse}
    For every integer $r\geq 2$, does there exist a function $f_r: \mathbb{N} \to \mathbb{N}$ that satisfies the following?:
    For all $r$-graph $F$ with maximum degree at most $\Delta$, the inequality
    \begin{equation*}
        s(F) \leq f_r(\Delta) e(F)
    \end{equation*}
    holds.
\end{problem}

To answer \Cref{prob:sparse} affirmatively, one possible approach that fits with our framework is showing that for every bounded degree hypergraph, there is a sufficiently small and sufficiently sparse dominating hypergraph that contains the given hypergraph. However, this direction seems hard to achieve because there are not many known dominating hypergraphs so far.

For further applications of our framework, we would like to emphasize the role of \Cref{prop:weakly-tensor}. Let $G_1$ and $G_2$ be $d_1$-regular and $d_2$-regular graph. Then $G_1 \otimes G_2$ is a $d_1d_2$-regular graph on $v(G_1)v(G_2)$ vertices. Also, every eigenvalue of $G_1 \otimes G_2$ is in form of $\lambda \mu$, where $\lambda$ and $\mu$ are eigenvalues of $G_1$ and $G_2$, respectively. Since weakly norming graphs are closed under the tensor product, if we iteratively take a tensor product with a balanced complete bipartite graph and even cycles, then one can obtain a sufficiently sparse bipartite expander, which is dominating. Since expander graphs contain large sparse graphs as subgraphs, this argument, together with a method demonstrated in \Cref{sec:proof-applications}, we believe that one can obtain a reasonable bound on Sidorenko exponents of sparse $3$-partite $3$-graphs.

Not only did we provide a general framework for bounding the Sidorenko exponents, but we also discovered a new connection between the upper bounds of Sidorenko exponents and the upper bounds of extremal numbers, which gives another motivation for the research on Sidorenko exponents. In particular, \Cref{thm:main-extremal} generalizes the classical K\H{o}v\'{a}ri--S\'{o}s--Tur\'{a}n theorem and its hypergraph analogue, and also the proof of \Cref{thm:main-extremal} gives an alternative proof of \Cref{thm:KST}.

Indeed, \Cref{thm:KST} is tight in some sense. For all $r \geq 2$ and $t_1, \dots, t_{r-1}$, there is a positive integer $t' > 0$ such that for all $t > t'$, we have
\begin{equation}\label{eq:KST-lowerbound}
    \mathrm{ex}(n, K^{(r)}_{t_1, \dots, t_{r-1}, t}) \geq \Omega_{t_1,\dots, t_{r-1}} \left(n^{r - \frac{1}{\prod_{i\in [r-1]}t_i}} \right),
\end{equation}
which matches the bound provided by \Cref{thm:KST}.
The graph case $(r = 2)$ of \eqref{eq:KST-lowerbound} was proved by Koll\'{a}r, R\'{o}nayi, and Szab\'{o}~\cite{Kollar-Ronyai-Szabo} by introducing the notion of norm graphs, and this was further refined by Alon, R\'{o}nayi, and Szab\'{o}~\cite{Alon-Ronyai-Szabo}. Fairly recently, Bukh~\cite{Bukh} significantly improved the threshold of $t$ that \eqref{eq:KST-lowerbound} holds by using a random algebraic construction.

For higher uniformities, Ma, Yuan, and Zhang~\cite{Ma-Yuan-Zhang} proved that \eqref{eq:KST-lowerbound} holds for sufficiently large $t$. Later, an explicit upper bound of the threshold for $t$ that makes \eqref{eq:KST-lowerbound} true was obtained by Pohoata and Zakharov~\cite{norm-hypergraphs} by generalizing norm graphs to hypergraphs.

Hence, it is natural to conjecture that \Cref{thm:main-extremal} is the best possible for sufficiently large $t$.

\begin{conjecture}\label{conj:KST-type-lowerbound}
    Let $r \geq 2$ be an integer and $F$ be an $(r-1)$-graph. Then there exists a positive integer $C_F$ such that the following holds for all $t > C_F$.
    \begin{equation*}
        \mathrm{ex}(n, F(t)) \geq \Omega_F\left( n^{r - \frac{1}{s(F)}} \right).
    \end{equation*}
\end{conjecture}

One of the smallest interesting unknown cases of \Cref{conj:KST-type-lowerbound} is $C_6(t)$. Since $C_6$ is a Sidorenko, by \Cref{thm:main-extremal}, we have
\begin{equation*}
    \mathrm{ex}(n, C_6(t)) \leq O_t\left(n^{\frac{17}{6}}\right).
\end{equation*}

Determining whether $C_6(t)$ satisfies \Cref{conj:KST-type-lowerbound} would already be an interesting problem.

\begin{conjecture}
    There is a constant $C > 0$ such that the following holds for all $t \geq C$.
    \begin{equation*}
        \mathrm{ex}(n, C_6(t)) \geq \Omega\left(n^{\frac{17}{6}}\right).
    \end{equation*}
\end{conjecture}

\section*{Acknowledgement}
The author is supported by the National Research Foundation of Korea (NRF) grant funded by the Korea government(MSIT) No. RS-2023-00210430, and the Institute for Basic Science (IBS-R029-C4).

The author would like to thank Seonghyuk Im for the valuable discussions.


\renewcommand{\MR}[1]{}
\renewcommand{\MRhref}[2]{%
  \href{http://www.ams.org/mathscinet-getitem?mr=#1}{#2}
}

    \bibliographystyle{amsplain_initials_nobysame}
    \bibliography{bibfile}


{\small
\bigskip\bigskip\bigskip\obeylines\parindent=0pt
Hyunwoo Lee
Department of Mathematical Sciences, KAIST
Extremal Combinatorics and Probability Group (ECOPRO), Institute for Basic Science (IBS)
Daejeon
South Korea
{\it E-mail addresses}: \tt hyunwoo.lee@kaist.ac.kr}

\end{document}